\documentclass[12pt,reqno]{amsart}
\usepackage[left=80pt,right=80pt]{geometry}
\usepackage[usenames]{color}
\usepackage{amsmath}
\usepackage{amssymb}
\usepackage{amsthm}
\usepackage{enumitem}
\usepackage{float}
\usepackage{graphicx}
\usepackage{dsfont}
\usepackage{array}
\usepackage{calc}

\usepackage{tikz}
\usepackage{subcaption}
\usepackage{cases}
\usepackage{hyperref,url}
\hypersetup{
	colorlinks=true,
  linkcolor=black,          
  citecolor=black,         
  filecolor=black,      
  urlcolor=black           
}

\newlength{\mycolwidth}
\newcolumntype{C}{>{\centering\arraybackslash}p{\mycolwidth}}  

\newcommand{\vpad}[1]{\rule[-1em]{0pt}{3em}#1}

\newtheorem{prop}{Proposition}

\newtheorem{theorem}[prop]{Theorem}

\newtheorem{corollary}[prop]{Corollary}
\theoremstyle{definition}

\newcommand{\seqnum}[1]{\href{https://oeis.org/#1}{\rm \underline{#1}}}
\allowdisplaybreaks
\newcommand{\mylabel}[2]{#2\def\@currentlabel{#2}\label{#1}}
\makeatletter
\newcommand{\vast}{\bBigg@{4}}
\newcommand{\Vast}{\bBigg@{5}}
\makeatother
\begin{document}
\tikzset{mystyle/.style={matrix of nodes,
        nodes in empty cells,
        row 1/.style={nodes={draw=none}},
        row sep=-\pgflinewidth,
        column sep=-\pgflinewidth,
        nodes={draw,minimum width=1cm,minimum height=1cmanchor=center}}}
\tikzset{mystyleb/.style={matrix of nodes,
        nodes in empty cells,
        row sep=-\pgflinewidth,
        column sep=-\pgflinewidth,
        nodes={draw,minimum width=1cm,minimum height=1cmanchor=center}}}

\title{Even-up words and their variants}

\author[SELA FRIED]{Sela Fried$^*$}
\thanks{$^*$Department of Computer Science, Israel Academic College,
52275 Ramat Gan, Israel.
\\
\href{mailto:friedsela@gmail.com}{\tt friedsela@gmail.com}}

\begin{abstract}
Inspired by OEIS sequence \seqnum{A377912}, which consists of the nonnegative integers in which every even digit (except possibly the last) is immediately followed by a strictly larger digit, we define even-up and odd-up words over an alphabet of size~$k$ via similar constraints. We introduce and analyze weak and cyclic variants of these words, deriving explicit generating functions for all eight resulting classes. We then study Catalan words under analogous restrictions. Our results provide new combinatorial interpretations for many integer sequences, including the Motzkin numbers, the Riordan numbers, and the generalized Catalan numbers.
\bigskip

\noindent \textbf{Keywords:} Motzkin number, Rirodan number, generalized Catalan number, generating function, Catalan word.
\smallskip

\noindent
\textbf{Math.~Subj.~Class.:} 05A05, 05A15.
\end{abstract}

\maketitle

\section{Introduction} 
The sequence \seqnum{A377912} in the On-Line Encyclopedia of Integer Sequences (OEIS) \cite{SL} consists of the nonnegative integers in which each even digit (except possibly the last) is immediately followed by a strictly larger digit. This digit-based constraint inspired the definition of even-up words over a finite alphabet. Let $k\geq 2$ and $n\geq 0$ be two integers, and set $[k] = \{1,2,\ldots,k\}$. A word over $k$ of length $n$ is simply an element of $[k]^n$. An even-up word over $k$ of length $n$ is a word $w_1\cdots w_n\in[k]^n$ such that for every $i\in[n-1]$, if $w_i$ is even, then $w_{i+1}>w_i$. Odd-up words are defined analogously. Replacing the strict inequality by a non-strict one gives rise to weakly even-up and weakly odd-up words. In each of these four cases, we add the prefix `cyclic', if, in addition, the constraint in question applies also to $i=n$, where we define $w_{n+1} = w_1$. For example, the word $342121$ is cyclic odd-up (and therefore, of course, also weakly odd-up and cyclic weakly odd-up). 

This work is mainly concerned with the enumeration of these eight word classes. For each class we derive an explicit generating function. We then turn to Catalan words and enumerate those that are also (weakly) even/odd-up. Catalan words of length $n$ are words $w_{1}\cdots w_{n}\in[n]^n$ satisfying $w_{1}=1$ and $w_{i+1}\leq w_{i}+1$, for every $i\in [n-1]$. They are so named because they are enumerated by the Catalan numbers (e.g., \cite{St}). 

Our analysis yields new combinatorial interpretations for many existing integer sequences in the OEIS. For instance, the number of weakly even-up Catalan words of length n, which end with an odd number is given by the $n+1$th generalized Catalan number (see Theorem \ref{t0} and its proof).

Our work falls within the line of research on restricted words, which has a long and rich history in enumerative combinatorics. For example, Carlitz and Scoville \cite{C} studied up-down words, and, more recently, Knopfmacher et al.\ \cite{K} studied staircase words. A natural extension of this framework involves imposing additional constraints on already restricted classes. Among the most studied in this regard are Catalan words. For instance, Baril and Ramirez \cite{BR} investigated the enumeration of Catalan words avoiding ordered pairs of relations. In this context, our focus on Catalan words is a natural continuation of this line of inquiry. Moreover, uncovering new combinatorial interpretations for classical integer sequences remains a central pursuit in the field. As an illustration, Mansour and Shattuck \cite{MS} recently demonstrated that the generalized Catalan numbers also enumerate set partitions avoiding a pair of classical patterns of length four.

\section{Main results}

Theorems stated without proof in the sequel are established using arguments similar in nature to those provided in the detailed proofs. Each theorem is accompanied by a table illustrating the enumerated restricted word class for small values of $k$ and $n$.

\begin{theorem}\label{t102}
Let $A_k(x) = \sum_{n\geq 0} a_{k,n}x^n$ be the generating function for the number of even-up words. Then
\[A_k(x)=\frac{(x+1)^{\left\lfloor \frac{k}{2}\right\rfloor }}{2-(x+1)^{\left\lfloor \frac{k+1}{2}\right\rfloor }}.\] 
\end{theorem}

\begin{proof}
For $i\in [k]$ let $A_{k,i}(x) = \sum_{n\geq 1} a_{k,i,n}x^n$ be the generating function for the number of even-up words which end with $i$. We have 
\[a_{k,i,n+1}=\sum_{\substack{j\in[i-1]\\
j\text{ even}
}
}a_{k,j,n}+\sum_{\substack{j\in[k]\\
j\text{ odd}
}
}a_{k,j,n}.\] Multiplying both sides of the equation by $x^{n+1}$ and summing over $n\geq 1$, we obtain \[A_{k,i}(x)=x+x\left(\sum_{\substack{j\in[i-1]\\
j\text{ even}
}
}A_{k,j}(x)+\sum_{\substack{j\in[k]\\
j\text{ odd}
}
}A_{k,j}(x)\right).\] The corresponding coefficient matrix is then given by $I-xM$, where \[(M)_{ij}=
\begin{cases}
0, & \textnormal{if } i\leq j\text{ and }j\text{ is even};\\
1, & \text{otherwise}.
\end{cases}\] We claim that the solution of the system is given by \[A_{k,i}(x)=\frac{x(x+1)^{\lfloor \frac{i-1}{2}\rfloor}}{2-(x+1)^{\lfloor \frac{k+1}{2}\rfloor}},\;i\in[k].\]
Indeed, let $i\in[k]$ and assume that $k$ is even. We have 
\begin{align}
&A_{k,i}(x)-x\sum_{j=1}^{k}(M)_{ij}A_{k,j}(x)\nonumber\\&=A_{k,i}(x)-x\sum_{j=1}^{k}A_{k,j}(x)+x\sum_{j=\left\lceil \frac{i}{2}\right\rceil }^{\frac{k}{2}}A_{k,2j}(x)\nonumber\\
&=\frac{x}{2-(x+1)^{\frac{k}{2}}}\left((x+1)^{\left\lfloor \frac{i-1}{2}\right\rfloor }-x\sum_{j=1}^{k}(x+1)^{\left\lfloor \frac{j-1}{2}\right\rfloor }+x\sum_{j=\left\lceil \frac{i}{2}\right\rceil }^{\frac{k}{2}}(x+1)^{\left\lfloor \frac{2j-1}{2}\right\rfloor }\right)\nonumber\\
&=\frac{x}{2-(x+1)^{\frac{k}{2}}}\left((x+1)^{\left\lfloor \frac{i-1}{2}\right\rfloor }-x\sum_{j=0}^{k-1}(x+1)^{\left\lfloor \frac{j}{2}\right\rfloor }+x\sum_{j=\left\lceil \frac{i}{2}\right\rceil }^{\frac{k}{2}}(x+1)^{j-1}\right)\nonumber\\
&=\frac{x}{2-(x+1)^{\frac{k}{2}}}\left((x+1)^{\left\lfloor \frac{i-1}{2}\right\rfloor }-x\sum_{j=0}^{\frac{k}{2}-1}(x+1)^{j}-x\sum_{j=0}^{\left\lceil \frac{i}{2}\right\rceil -2}(x+1)^{j}\right)\nonumber\\
&=\frac{x}{2-(x+1)^{\frac{k}{2}}}\left(2-(x+1)^{\frac{k}{2}}-(x+1)^{\left\lceil \frac{i}{2}\right\rceil -1}+(x+1)^{\left\lfloor \frac{i-1}{2}\right\rfloor }\right)\nonumber\\
&=x.\nonumber
\end{align} To obtain $A_k(x)$, we sum $A_{k,i}(x)$ over $i\in[k]$. Hence,
\begin{align}
A_{k}(x)	&=1+\sum_{i=1}^{k}\frac{x(x+1)^{\left\lfloor \frac{i-1}{2}\right\rfloor }}{2-(x+1)^{\frac{k}{2}}}\nonumber\\
&=1+\frac{2x}{2-(x+1)^{\frac{k}{2}}}\sum_{i=0}^{\frac{k}{2}-1}(x+1)^{i}\nonumber\\
&=1-\frac{2\left(1-(x+1)^{\frac{k}{2}}\right)}{2-(x+1)^{\frac{k}{2}}}\nonumber\\
&=\frac{(x+1)^{\frac{k}{2}}}{2-(x+1)^{\frac{k}{2}}} \nonumber.
\end{align}
The case of odd $k$ is similar.
\end{proof}

\begin{table}[H]
\begin{center}
{\renewcommand{\arraystretch}{1.3}
\begin{tabular}{| c|c|c|c|c|c|c|c|c|c|c|c |c|}
 \hline
   $k/ n$& $0$ &$1$ &$2$ &$3$ &$4$ &$5$ &$6$ &$7$ &$8$ &$9$ &$10$ & OEIS\\ [0.5ex]
 \hline
 $1$&$1 $ &$1 $ &$1 $ &$1 $ &$1 $ &$1 $ &$1 $ &$1 $ &$1 $ &$1 $ &$1 $ & trivial  \\ \hline
 $2$&$1 $ &$2 $ &$2 $ &$2 $ &$2 $ &$2 $ &$2 $ &$2 $ &$2 $ &$2 $ &$2 $ &trivial \\ \hline
 $3$&$1 $ &$3 $ &$7 $ &$17 $ &$41 $ &$99 $ &$239 $ &$577 $ &$1393 $ &$3363 $ &$8119 $ &$\seqnum{A001333}$ \\ \hline
 $4$&$1 $ &$4 $ &$10 $ &$24 $ &$58 $ &$140 $ &$338 $ &$816 $ &$1970 $ &$4756 $ &$11482 $ & $\seqnum{A052542} $\\ \hline 
 $5$ & $1 $ &$5 $ &$19 $ &$73 $ &$281 $ &$1081 $ &$4159 $ &$16001 $ &$61561 $ &$236845 $ &$911219 $ &\seqnum{A377314}\\ \hline
 $6$&$1 $ &$6 $ &$24 $ &$92 $ &$354 $ &$1362 $ &$5240 $ &$20160 $ &$77562 $ &$298406 $ &$1148064 $ &$\seqnum{A108368}$ 
 \\ \hline
\end{tabular}
\caption{Number of even-up words.}\label{tablea2}}
\end{center}
\end{table}

\begin{theorem}
Let $B_k(x) = \sum_{n\geq 0} b_{k,n}x^n$ be the generating function for the number of odd-up words. Then
\[B_k(x)=\frac{(x+1)^{\left\lfloor \frac{k+1}{2}\right\rfloor }}{x+2-(x+1)^{\left\lfloor \frac{k+2}{2}\right\rfloor }}.\] 
\end{theorem}

\begin{table}[H]
\begin{center}
{\renewcommand{\arraystretch}{1.3}
\begin{tabular}{| c|c|c|c|c|c|c|c|c|c|c|c |c|}
 \hline
   $k/ n$& $0$ &$1$ &$2$ &$3$ &$4$ &$5$ &$6$ &$7$ &$8$ &$9$ &$10$ & OEIS\\ [0.5ex]
 \hline

$1$&$1 $ &$1 $ &$0 $ &$0 $ &$0 $ &$0 $ &$0 $ &$0 $ &$0 $ &$0 $ &$0 $  &trivial\\ \hline

$2$&$1 $ &$2 $ &$3 $ &$5 $ &$8 $ &$13 $ &$21 $ &$34 $ &$55 $ &$89 $ &$144 $ &$\seqnum{A000045}$\\ \hline

$3$&$1 $ &$3 $ &$5 $ &$8 $ &$13 $ &$21 $ &$34 $ &$55 $ &$89 $ &$144 $ &$233 $ &$\seqnum{A000045}$\\ \hline

 $4$&$1 $ &$4 $ &$12 $ &$37 $ &$114 $ &$351 $ &$1081 $ &$3329 $ &$10252 $ &$31572 $ &$97229$&$\seqnum{A099098}$  \\ \hline

 $5$&$1 $ &$5 $ &$16 $ &$49 $ &$151 $ &$465 $ &$1432 $ &$4410 $ &$13581 $ &$41824 $ &$128801 $ &$\seqnum{A334293}$ \\ \hline

  $6$&$1 $ &$6 $ &$27 $ &$122 $ &$553 $ &$2505 $ &$11348 $ &$51408 $ &$232885 $ &$1055000 $ &$4779290 $ &- \\ \hline
\end{tabular}
\caption{Number of odd-up words.}\label{tablea3}}
\end{center}
\end{table}

\begin{theorem}
Let $F_k(x)=\sum_{n\geq 0} f_{k,n}x^n$ be the generating function for the number of cyclic even-up words. Then 
\[
F_k(x) =1+x\left(\left\lfloor \frac{k}{2}\right\rfloor -\left\lfloor \frac{k+1}{2}\right\rfloor\frac{ (x+1)^{\left\lfloor \frac{k-1}{2}\right\rfloor }}{(x+1)^{\left\lfloor \frac{k+1}{2}\right\rfloor }-2}\right)
\] 
\end{theorem}

\begin{proof}
It will be convenient to modify $A_k(x)$ from Theorem \ref{t102} and replace it with $A_k(x)+1/x$. We count the number of cyclic even-up words over $k+1$ according to the number of occurrences and positions of the letter $k+1$. Notice that if $k+1$ is even, then the letter $k+1$ cannot be used. Thus, in this case, $F_{k+1}(x) = F_k(x) + x$. Now, assume that $k+1$ is odd. Then the letter $k+1$ can be placed at any place. There are $f_{k,n}$ cyclic even-up words without any $k+1$ and $nf_{k,n-1}$ such words with exactly one occurrence of $k+1$. Now assume that there are at least two occurrences of $k+1$ and denote by $i$ and $j$ the smallest and largest index of these $k+1$s, respectively. There are $a_{k+1,j-i-1}$ possibilities for the subword $w_{i+1}\cdots w_{j-1}$ and $a_{k,n-(j-i)-1}$ possibilities for the (cyclic) subword $w_{j+1}\cdots w_nw_1\cdots w_{i-1}$. It follows that 
\begin{align}
f_{k+1,n}	&=f_{k,n}+na_{k,n-1}+\sum_{i=1}^{n-1}\sum_{j=i+1}^{n}a_{k+1,j-i-1}a_{k,n-(j-i)-1}\nonumber\\
&=f_{k,n}+\sum_{i=1}^{n}ia_{k+1,n-i-1}a_{k,i-1}.\nonumber
\end{align} Multiplying both sides of the equation by $x^n$ and summing over $n\geq 1$, we obtain \[F_{k+1}(x)=F_{k}(x)+x^{2}A_{k+1}(x)\frac{d}{dx}\left(xA_{k}(x)\right).\] Thus,
\begin{align}
&F_{k}(x)\nonumber\\
&=1+\left\lfloor \frac{k}{2}\right\rfloor x+x^{2}\sum_{i=0}^{\left\lfloor \frac{k-1}{2}\right\rfloor }A_{2i+1}(x)\frac{d}{dx}(xA_{2i}(x))\nonumber\\
&=1+\left\lfloor \frac{k}{2}\right\rfloor x+x^{2}\sum_{i=0}^{\left\lfloor \frac{k-1}{2}\right\rfloor }\left(\frac{(x+1)^{i}}{2-(x+1)^{i+1}}+\frac{1}{x}\right)\frac{d}{dx}\left(\frac{x(x+1)^{i}}{2-(x+1)^{i}}+1\right)\nonumber\\
&=1+\left\lfloor \frac{k}{2}\right\rfloor x+x\sum_{i=0}^{\left\lfloor \frac{k-1}{2}\right\rfloor }\frac{(x+1)^{3i}-4(x+1)^{2i}-2ix(x+1)^{2i-1}+4(x+1)^{i}+4ix(x+1)^{i-1}}{(2-(x+1)^{i})^{2}(2-(x+1)^{i+1})}\nonumber.
\end{align} 
We now use the method of creative telescoping (cf.\ \cite{Z}). To this end, define \[G(i) = (i+1)\frac{(x+1)^{i}}{2-(x+1)^{i+1}}.\] It is straightforward to verify that $G(i)-G(i-1)$ is equal to the $i$th summand in the last sum. Thus, 
\begin{align}
F_k(x)&=1+\left\lfloor \frac{k}{2}\right\rfloor x+x\sum_{i=0}^{\left\lfloor \frac{k-1}{2}\right\rfloor }(G(i)-G(i-1))\nonumber\\
&=1+\left\lfloor \frac{k}{2}\right\rfloor x+x\left(G\left({\left\lfloor \frac{k-1}{2}\right\rfloor }\right)-\overbrace{G(-1)}^{=0}\right)\nonumber\\
&=1+x\left(\left\lfloor \frac{k}{2}\right\rfloor +\left\lfloor \frac{k+1}{2}\right\rfloor\frac{ (x+1)^{\left\lfloor \frac{k-1}{2}\right\rfloor }}{2-(x+1)^{\left\lfloor \frac{k+1}{2}\right\rfloor }}\right).\nonumber\qedhere
\end{align}
\end{proof}

\begin{table}[H]
\begin{center}
{\renewcommand{\arraystretch}{1.3}
\begin{tabular}{| c|c|c|c|c|c|c|c|c|c|c|c |c|}
 \hline
   $k/ n$& $0$ &$1$ &$2$ &$3$ &$4$ &$5$ &$6$ &$7$ &$8$ &$9$ &$10$ & OEIS\\ [0.5ex]
 \hline

$1$&   $1$&$ 1$&$ 1$&$ 1$&$ 1$&$ 1$&$ 1$&$ 1$&$ 1$&$ 1$&$ 1     $ &trivial\\ \hline

$2$& $  1$&$ 2$&$ 1$&$ 1$&$ 1$&$ 1$&$ 1$&$ 1$&$ 1$&$ 1$&$ 1    $ &trivial\\ \hline

$3$& $  1$&$ 3$&$ 6$&$ 14$&$ 34$&$ 82$&$ 198$&$ 478$&$ 1154$&$ 2786$&$ 6726   $ &\seqnum{A002203}\\ \hline

$4$& $  1$&$ 4$&$ 6$&$ 14$&$ 34$&$ 82$&$ 198$&$ 478$&$ 1154$&$ 2786$&$ 6726   $ &\seqnum{A002203} \\ \hline

$5$& $  1$&$ 5$&$ 15$&$ 57$&$ 219$&$ 843$&$ 3243$&$ 12477$&$ 48003$&$ 184683$&$ 710535   $ &- \\ \hline

$6$& $  1$&$ 6$&$ 15$&$ 57$&$ 219$&$ 843$&$ 3243$&$ 12477$&$ 48003$&$ 184683$&$ 710535   $ &- \\ \hline 
\end{tabular}
\caption{Number of cyclic even-up words.}\label{tablea4}}
\end{center}
\end{table}

\begin{theorem}
Let $F_k(x)=\sum_{n\geq 0} f_{k,n}x^n$ be the generating function for the number of cyclic odd-up words. Then 
\[
F_k(x) =1+x\left(\left\lfloor \frac{k+1}{2}\right\rfloor +\frac{\left\lfloor \frac{k+2}{2}\right\rfloor (x+1)^{\left\lfloor \frac{k}{2}\right\rfloor }-1}{x+2-(x+1)^{\left\lfloor \frac{k+2}{2}\right\rfloor }}\right).
\]
\end{theorem}

\begin{table}[H]
\begin{center}
{\renewcommand{\arraystretch}{1.3}
\begin{tabular}{| c|c|c|c|c|c|c|c|c|c|c|c |c|}
 \hline
   $k/ n$& $0$ &$1$ &$2$ &$3$ &$4$ &$5$ &$6$ &$7$ &$8$ &$9$ &$10$ & OEIS\\ [0.5ex]
 \hline

    $1$& $     1$&$ 1$&$ 0$&$ 0$&$ 0$&$ 0$&$ 0$&$ 0$&$ 0$&$ 0$&$ 0    $ & trivial\\ \hline
    $2$& $    1$&$ 2$&$ 3$&$ 4$&$ 7$&$ 11$&$ 18$&$ 29$&$ 47$&$ 76$&$ 123  $ & \seqnum{A000032}\\ \hline

    $3$& $    1$&$ 3$&$ 3$&$ 4$&$ 7$&$ 11$&$ 18$&$ 29$&$ 47$&$ 76$&$ 123 $ & \seqnum{A000032}\\ \hline

    $4$& $  1$&$ 4$&$ 10$&$ 29$&$ 90$&$ 277$&$ 853$&$ 2627$&$ 8090$&$ 24914$&$ 76725 $ & -\\ \hline

    $5$& $  1$&$ 5$&$ 10$&$ 29$&$ 90$&$ 277$&$ 853$&$ 2627$&$ 8090$&$ 24914$&$ 76725 $ & -\\ \hline

    $6$& $  1$&$ 6$&$ 21$&$ 93$&$ 421$&$ 1908$&$ 8643$&$ 39154$&$ 177373$&$ 803523$&$ 3640066 $ & - \\ \hline 
\end{tabular}
\caption{Number of cyclic odd-up words.}\label{tablea5}}
\end{center}
\end{table}

\begin{theorem}
Let $A_k(x) = \sum_{n\geq 0} a_{k,n}x^n$ be the generating function for the number of weakly even-up words. Then
\[A_k(x)=\begin{cases}
\frac{1}{(1-x)^{\frac{k+2}{2}}+(1-x)^{\frac{k+2}{2}-1}+x-1}, & \textnormal{if \ensuremath{k} is even};\\
\frac{1}{(1-x)^{\frac{k+1}{2}}+(1-x)^{\frac{k+1}{2}-1}-1}, & \textnormal{if \ensuremath{k} is odd}.
\end{cases}\] 
\end{theorem}
\begin{table}[H]
\begin{center}
{\renewcommand{\arraystretch}{1.3}
\begin{tabular}{| c|c|c|c|c|c|c|c|c|c|c|c |c|}
 \hline
   $k/ n$& $0$ &$1$ &$2$ &$3$ &$4$ &$5$ &$6$ &$7$ &$8$ &$9$ &$10$ & OEIS\\ [0.5ex]
 \hline
$1$&$1$&$1$&$1$&$1$&$1$&$1$&$1$&$1$&$1$&$1$&$1$&trivial\\ \hline 
$2$&$1$&$2$&$3$&$4$&$5$&$6$&$7$&$8$&$9$&$10$&$11$&trivial\\ \hline 
$3$&$1$&$3$&$8$&$21$&$55$&$144$&$377$&$987$&$2584$&$6765$&$17711$&\seqnum{A001906}\\ \hline 
$4$&$1$&$4$&$12$&$33$&$88$&$232$&$609$&$1596$&$4180$&$10945$&$28656$&\seqnum{A027941}\\ \hline 
$5$&$1$&$5$&$21$&$86$&$351$&$1432$&$5842$&$23833$&$97229$&$396655$&$1618192$&\seqnum{A012814}\\ \hline 
$6$&$1$&$6$&$27$&$113$&$464$&$1896$&$7738$&$31571$&$128800$&$525455$&$2143647$&\seqnum{A176476}\\ \hline 
\end{tabular}
\caption{Number of weakly even-up words.}\label{tablea6}}
\end{center}
\end{table}

\begin{theorem}
Let $B_k(x) = \sum_{n\geq 0} b_{k,n}x^n$ be the generating function for the number of weakly odd-up words. Then
\[B_k(x)=\begin{cases}
\frac{1}{2(1-x)^{\frac{k}{2}}-1}, & \textnormal{if \ensuremath{k} is even};\\
\frac{1}{2(1-x)^{\frac{k+1}{2}}+x-1}, & \textnormal{if \ensuremath{k} is odd}.
\end{cases}\] 
\end{theorem}

\begin{table}[H]
\begin{center}
{\renewcommand{\arraystretch}{1.3}
\begin{tabular}{| c|c|c|c|c|c|c|c|c|c|c|c |c|}
 \hline
   $k/ n$& $0$ &$1$ &$2$ &$3$ &$4$ &$5$ &$6$ &$7$ &$8$ &$9$ &$10$ & OEIS\\ [0.5ex]
 \hline
$1$&$1$&$1$&$1$&$1$&$1$&$1$&$1$&$1$&$1$&$1$&$1$&trivial\\ \hline 
$2$&$1$&$2$&$4$&$8$&$16$&$32$&$64$&$128$&$256$&$512$&$1024$&\seqnum{A000079}\\ \hline 
$3$&$1$&$3$&$7$&$15$&$31$&$63$&$127$&$255$&$511$&$1023$&$2047$&\seqnum{A000225}\\ \hline 
$4$&$1$&$4$&$14$&$48$&$164$&$560$&$1912$&$6528$&$22288$&$76096$&$259808$&\seqnum{A007070}\\ \hline 
$5$&$1$&$5$&$19$&$67$&$231$&$791$&$2703$&$9231$&$31519$&$107615$&$367423$&\seqnum{A035344}\\ \hline 
$6$&$1$&$6$&$30$&$146$&$708$&$3432$&$16636$&$80640$&$390888$&$1894760$&$9184512$&\seqnum{A145839}\\ \hline 
\end{tabular}
\caption{Number of weakly odd-up words.}\label{tablea7}}
\end{center}
\end{table}

\begin{theorem}
Let $F_k(x)=\sum_{n\geq 0} f_{k,n}x^n$ be the generating function for the number of cyclic weakly even-up words. Then 
\[
F_k(x) = \left\lfloor \frac{k}{2}\right\rfloor \frac{x}{1-x}+\frac{2(1-x)^{\left\lceil \frac{k}{2}\right\rceil }+\left\lceil \frac{k}{2}\right\rceil x-1}{(2-x)(1-x)^{\left\lceil \frac{k}{2}\right\rceil }+x-1}.
\]
\end{theorem}

\begin{table}[H]
\begin{center}
{\renewcommand{\arraystretch}{1.3}
\begin{tabular}{| c|c|c|c|c|c|c|c|c|c|c|c |c|}
 \hline
   $k/ n$& $0$ &$1$ &$2$ &$3$ &$4$ &$5$ &$6$ &$7$ &$8$ &$9$ &$10$ & OEIS\\ [0.5ex]
 \hline
$1$&$1$&$1$&$1$&$1$&$1$&$1$&$1$&$1$&$1$&$1$&$1$&trivial\\ \hline 
$2$&$1$&$2$&$2$&$2$&$2$&$2$&$2$&$2$&$2$&$2$&$2$&trivial\\ \hline 
$3$&$1$&$3$&$7$&$18$&$47$&$123$&$322$&$843$&$2207$&$5778$&$15127$&\seqnum{A005248}\\ \hline 
$4$&$1$&$4$&$8$&$19$&$48$&$124$&$323$&$844$&$2208$&$5779$&$15128$&\seqnum{A065034}\\ \hline 
$5$&$1$&$5$&$17$&$68$&$277$&$1130$&$4610$&$18807$&$76725$&$313007$&$1276942$&-\\ \hline 
$6$&$1$&$6$&$18$&$69$&$278$&$1131$&$4611$&$18808$&$76726$&$313008$&$1276943$&-\\ \hline 
\end{tabular}
\caption{Number of cyclic weakly even-up words.}\label{tablea8}}
\end{center}
\end{table}

\begin{theorem}
Let $F_k(x)=\sum_{n\geq 0} f_{k,n}x^n$ be the generating function for the number of cyclic weakly odd-up words. Then 
\[
F_k(x) =\left\lceil \frac{k}{2}\right\rceil \frac{x}{1-x}+\frac{2(1-x)^{\left\lceil \frac{k+1}{2}\right\rceil }+\left\lceil \frac{k+2}{2}\right\rceil x-1}{2(1-x)^{\left\lceil \frac{k+1}{2}\right\rceil -1}+x-1}.
\]
\end{theorem}

\begin{table}[H]
\begin{center}
{\renewcommand{\arraystretch}{1.3}
\begin{tabular}{| c|c|c|c|c|c|c|c|c|c|c|c |c|}
 \hline
   $k/ n$& $0$ &$1$ &$2$ &$3$ &$4$ &$5$ &$6$ &$7$ &$8$ &$9$ &$10$ & OEIS\\ [0.5ex]
 \hline
$1$&$1$&$1$&$1$&$1$&$1$&$1$&$1$&$1$&$1$&$1$&$1$&trivial\\ \hline 
$2$&$1$&$2$&$4$&$8$&$16$&$32$&$64$&$128$&$256$&$512$&$1024$&\seqnum{A000079}\\ \hline 
$3$&$1$&$3$&$5$&$9$&$17$&$33$&$65$&$129$&$257$&$513$&$1025$&\seqnum{A000051}\\ \hline 
$4$&$1$&$4$&$12$&$40$&$136$&$464$&$1584$&$5408$&$18464$&$63040$&$215232$&\seqnum{A056236}\\ \hline 
$5$&$1$&$5$&$13$&$41$&$137$&$465$&$1585$&$5409$&$18465$&$63041$&$215233$&-\\ \hline 
$6$&$1$&$6$&$24$&$114$&$552$&$2676$&$12972$&$62880$&$304800$&$1477464$&$7161744$&-\\ \hline  
\end{tabular}
\caption{Number of cyclic weakly odd-up words.}\label{tablea9}}
\end{center}
\end{table}

\begin{theorem}\label{t0}
Let $A(x)=\sum_{n\geq 0} a_nx^n$ and $B(x)=\sum_{n\geq 0} b_nx^n$ be the generating functions for the number of weakly even-up and weakly odd-up Catalan words, respectively. Then 
\begin{align}
A(x)&=\frac{1-x-\sqrt{(1+x^{2})^{2}-4x}}{x},\nonumber\\
B(x)&=-\frac{x^{3}+2x^{2}+x-2+(x+2)\sqrt{(1+x^{2})^{2}-4x}}{2x}.    \nonumber
\end{align}
\end{theorem}

\begin{proof}
Let $A'(x)=\sum_{n\geq 0} a'_nx^n$ be the generating function for the number of weakly even-up Catalan words which end with an odd number and let $B'(x)=\sum_{n\geq 0} b'_nx^n$ be the generating function for the number of weakly odd-up Catalan words which end with an even number. For every $n\geq 2$ we then have
\begin{align}
a_{n}&=b_{n-1}+\sum_{i=2}^{n}b'_{i-2}a_{n-i+1},\nonumber\\
b_{n}&=a_{n-1}+\sum_{i=2}^{n}a'_{i-2}b_{n-i+1},\nonumber\\
a'_{n}&=b'_{n-1}+\sum_{i=2}^{n}b'_{i-2}a'_{n-i+1},\nonumber\\
b'_{n}&=a'_{n-1}+\sum_{i=2}^{n}a'_{i-2}b'_{n-i+1}.    \nonumber
\end{align} Multiplying each of these equations by $x^n$ and summing over $n\geq 2$, we then obtain 
\begin{align}
A(x)&=1+xB(x)+xB'(x)(A(x)-1),\nonumber\\
B(x)&=1+xA(x)+xA'(x)(B(x)-1),\nonumber\\
A'(x)&=1+xA'(x)B'(x),\nonumber\\
B'(x)&=1-x+xA'(x)B'(x).    \nonumber
\end{align} Solving the third and fourth equations for $A'(x)$ and $B'(x)$, we obtain
\begin{align}
A'(x)&=\frac{1+x^{2}-\sqrt{(1+x^{2})^{2}-4x}}{2x},\nonumber\\
B'(x)&=\frac{1-x^{2}-\sqrt{(1+x^{2})^{2}-4x}}{2x}.    \nonumber
\end{align} Substituting these into the first two equations and solving them for $A(x)$ and $B(x)$, we finally obtain the asserted formulas. 
\end{proof}

\begin{table}[H]
\centering
\resizebox{\textwidth}{!}{%
\begin{tabular}{|*{12}{C|}} 
\hline
\multicolumn{6}{|c|}{Generating function} & \multicolumn{6}{c|}{OEIS} \\
\hline
\multicolumn{6}{|c|}{\vpad{$\displaystyle \frac{1-x-\sqrt{(1+x^{2})^{2}-4x}}{x}$}} & \multicolumn{6}{c|}{$2\times $\seqnum{A025242} (generalized Catalan numbers)} \\
\hline
$n$ & 0 & 1 & 2 & 3 & 4 & 5 & 6 & 7 & 8 & 9 & 10 \\
\hline
$a_n$ & 1 & 1 & 2 & 4 & 10 & 26 & 70 & 194 & 550 & 1588 & 4654 \\
\hline
\multicolumn{6}{|c|}{Generating function} & \multicolumn{6}{c|}{OEIS} \\
\hline
\multicolumn{6}{|c|}{\vpad{$\displaystyle -\frac{x^{3}+2x^{2}+x-2+(x+2)\sqrt{(1+x^{2})^{2}-4x}}{2x}$}} & \multicolumn{6}{c|}{-} \\
\hline
$n$ & 0 & 1 & 2 & 3 & 4 & 5 & 6 & 7 & 8 & 9 & 10 \\
\hline
$b_n$ & 1 & 1 & 2 & 5 & 12 & 31 & 83 & 229 & 647 & 1863 & 5448 \\
\hline
\multicolumn{6}{|c|}{Generating function} & \multicolumn{6}{c|}{OEIS} \\
\hline
\multicolumn{6}{|c|}{\vpad{$\displaystyle \frac{1+x^{2}-\sqrt{(1+x^{2})^{2}-4x}}{2x}$}} & \multicolumn{6}{c|}{\seqnum{A025242} (generalized Catalan numbers)} \\
\hline
$n$ & 0 & 1 & 2 & 3 & 4 & 5 & 6 & 7 & 8 & 9 & 10 \\
\hline
$a'_n$ &1& 1& 1& 2& 5& 13& 35& 97& 275& 794& 2327 \\
\hline
\multicolumn{6}{|c|}{Generating function} & \multicolumn{6}{c|}{OEIS} \\
\hline
\multicolumn{6}{|c|}{\vpad{$\displaystyle \frac{1-x^{2}-\sqrt{(1+x^{2})^{2}-4x}}{2x}$}} & \multicolumn{6}{c|}{\seqnum{A025242} (generalized Catalan numbers)} \\
\hline
$n$ & 0 & 1 & 2 & 3 & 4 & 5 & 6 & 7 & 8 & 9 & 10 \\
\hline
$b'_n$ & 1& 0& 1& 2& 5& 13& 35& 97& 275& 794& 2327 \\
\hline
\end{tabular}}
\caption{The generating functions of Theorem \ref{t0}.}
\end{table}

\begin{theorem}\label{t1}
Let $A(x)=\sum_{n\geq 0} a_nx^n$ and $B(x)=\sum_{n\geq 0} b_nx^n$ be the generating functions for the number of even-up and odd-up Catalan words, respectively. Then 
\begin{align}
A(x)&=\frac{2-x}{2x}-\frac{(2+x)\sqrt{(1+x)(1-3x)}}{2x(1+x)},\nonumber\\
B(x)&=\frac{1-\sqrt{(1+x)(1-3x)}}{x(1+x)}.    \nonumber
\end{align}
\end{theorem}

\begin{table}[H]
\centering
\resizebox{\textwidth}{!}{%
\begin{tabular}{|*{12}{C|}} 
\hline
\multicolumn{6}{|c|}{Generating function} & \multicolumn{6}{c|}{OEIS} \\
\hline
\multicolumn{6}{|c|}{\vpad{$\displaystyle 
\frac{2-x}{2x}-\frac{(2+x)\sqrt{(1+x)(1-3x)}}{2x(1+x)}$}} & \multicolumn{6}{c|}{-} \\
\hline
$n$ & 0 & 1 & 2 & 3 & 4 & 5 & 6 & 7 & 8 & 9 & 10 \\
\hline
$a_n$ & 1& 1& 2& 3& 7& 15& 36& 87& 218& 555& 1438 \\
\hline
\multicolumn{6}{|c|}{Generating function} & \multicolumn{6}{c|}{OEIS} \\
\hline
\multicolumn{6}{|c|}{\vpad{$\displaystyle \frac{1-\sqrt{(1+x)(1-3x)}}{x(1+x)}$}} & \multicolumn{6}{c|}{\seqnum{A124791}} \\
\hline
$n$ & 0 & 1 & 2 & 3 & 4 & 5 & 6 & 7 & 8 & 9 & 10 \\
\hline
$b_n$ & 1& 1& 1& 3& 5& 13& 29& 73& 181& 465& 1205 \\
\hline
\multicolumn{6}{|c|}{Generating function} & \multicolumn{6}{c|}{OEIS} \\
\hline
\multicolumn{6}{|c|}{\vpad{$\displaystyle \frac{1+x-\sqrt{(1+x)(1-3x)}}{2x}$}} & \multicolumn{6}{c|}{\seqnum{A001006} (Motzkin numbers)} \\
\hline
$n$ & 0 & 1 & 2 & 3 & 4 & 5 & 6 & 7 & 8 & 9 & 10 \\
\hline
$a'_n$ &1& 1& 1& 2& 4& 9& 21& 51& 127& 323& 835 \\
\hline
\multicolumn{6}{|c|}{Generating function} & \multicolumn{6}{c|}{OEIS} \\
\hline
\multicolumn{6}{|c|}{\vpad{$\displaystyle \frac{1+x-\sqrt{(1+x)(1-3x)}}{2x(1+x)}$}} & \multicolumn{6}{c|}{\seqnum{A005043} (Riordan numbers)} \\
\hline
$n$ & 0 & 1 & 2 & 3 & 4 & 5 & 6 & 7 & 8 & 9 & 10 \\
\hline
$b'_n$ & 1& 0& 1& 1& 3& 6& 15& 36& 91& 232& 603 \\
\hline
\end{tabular}}
\caption{The generating functions of Theorem \ref{t1}.}
\end{table}

\end{document}